\documentclass[11pt]{amsart}
\usepackage{amssymb,amsthm,amsmath,epsfig,latexsym}
\usepackage{calc,times,verbatim}
\usepackage{geometry} %geometry package allows easy manipulation of margins, etc.
\begin{document}

% \mmbox enables macros to survive outside of $ ... $
\newcommand{\mmbox}[1]{\mbox{${#1}$}}
\newcommand{\proj}[1]{\mmbox{{\mathbb P}^{#1}}}
\newcommand{\Cr}{C^r(\Delta)}
\newcommand{\CR}{C^r(\hat\Delta)}
\newcommand{\affine}[1]{\mmbox{{\mathbb A}^{#1}}}
\newcommand{\Ann}[1]{\mmbox{{\rm Ann}({#1})}}
\newcommand{\caps}[3]{\mmbox{{#1}_{#2} \cap \ldots \cap {#1}_{#3}}}
\newcommand{\Proj}{{\mathbb P}}
\newcommand{\N}{{\mathbb N}}
\newcommand{\Z}{{\mathbb Z}}
\newcommand{\R}{{\mathbb R}}
\newcommand{\A}{{\mathcal{A}}}
\newcommand{\Tor}{\mathop{\rm Tor}\nolimits}
\newcommand{\Ext}{\mathop{\rm Ext}\nolimits}
\newcommand{\Hom}{\mathop{\rm Hom}\nolimits}
\newcommand{\im}{\mathop{\rm Im}\nolimits}
\newcommand{\rank}{\mathop{\rm rank}\nolimits}
\newcommand{\supp}{\mathop{\rm supp}\nolimits}
\newcommand{\arrow}[1]{\stackrel{#1}{\longrightarrow}}
\newcommand{\CB}{Cayley-Bacharach}
\newcommand{\coker}{\mathop{\rm coker}\nolimits}
\sloppy
\theoremstyle{plain}

\newtheorem*{thm*}{Theorem}
\newtheorem{defn0}{Definition}[section]
\newtheorem{prop0}[defn0]{Proposition}
\newtheorem{quest0}[defn0]{Question}
\newtheorem{thm0}[defn0]{Theorem}
\newtheorem{lem0}[defn0]{Lemma}
\newtheorem{corollary0}[defn0]{Corollary}
\newtheorem{example0}[defn0]{Example}
\newtheorem{remark0}[defn0]{Remark}
\newtheorem{conj0}[defn0]{Conjecture}

\newenvironment{defn}{\begin{defn0}}{\end{defn0}}
\newenvironment{conj}{\begin{conj0}}{\end{conj0}}
\newenvironment{prop}{\begin{prop0}}{\end{prop0}}
\newenvironment{quest}{\begin{quest0}}{\end{quest0}}
\newenvironment{thm}{\begin{thm0}}{\end{thm0}}
\newenvironment{lem}{\begin{lem0}}{\end{lem0}}
\newenvironment{cor}{\begin{corollary0}}{\end{corollary0}}
\newenvironment{exm}{\begin{example0}\rm}{\end{example0}}
\newenvironment{rem}{\begin{remark0}\rm}{\end{remark0}}

\newcommand{\defref}[1]{Definition~\ref{#1}}
\newcommand{\conjref}[1]{Conjecture~\ref{#1}}
\newcommand{\propref}[1]{Proposition~\ref{#1}}
\newcommand{\thmref}[1]{Theorem~\ref{#1}}
\newcommand{\lemref}[1]{Lemma~\ref{#1}}
\newcommand{\corref}[1]{Corollary~\ref{#1}}
\newcommand{\exref}[1]{Example~\ref{#1}}
\newcommand{\secref}[1]{Section~\ref{#1}}
\newcommand{\remref}[1]{Remark~\ref{#1}}
\newcommand{\questref}[1]{Question~\ref{#1}}

\newcommand{\std}{Gr\"{o}bner}
\newcommand{\jq}{J_{Q}}

\title{On some ideals with linear free resolutions}
\author{\c{S}tefan O. Toh\v{a}neanu}

. \keywords{linear free resolution, fold products, linear codes. \\
\indent Tohaneanu's address: Department of Mathematics, University of Idaho, Moscow, Idaho 83844-1103, USA, Email: tohaneanu@uidaho.edu.}

\begin{abstract}
Given $\Sigma\subset\mathbb K[x_1,\ldots,x_k]$, any finite collection of linear forms, some possibly proportional, and any $1\leq a\leq |\Sigma|$, it has been conjectured that $I_a(\Sigma)$, the ideal generated by all $a$-fold products of $\Sigma$, has linear graded free resolution. In this article we show the validity of this conjecture for two cases: the first one is when $a=d+1$ and $\Sigma$ is dual to the columns of a generating matrix of a linear code of minimum distance $d$; and the second one is when $k=3$ and $\Sigma$ defines a line arrangement in $\mathbb P^2$ (i.e., there are no proportional linear forms). For the second case we investigate what are the graded betti numbers of $I_a(\Sigma)$.
\end{abstract}

\maketitle

\section{Introduction}

Let $R:=\mathbb K[x_1,\ldots,x_k]$ be the ring of (homogeneous) polynomials with coefficients in a field $\mathbb K$, with the standard grading. Denote ${\frak m}:=\langle x_1,\ldots,x_k\rangle$ to be the irrelevant maximal ideal of $R$. Let $\ell_1,\ldots,\ell_n$ be linear forms in $R$, some possibly proportional, and denote this collection by $\Sigma=(\ell_1,\ldots,\ell_n)\subset R$; we need a notation to take into account the fact that some of these linear forms are proportional. For $\ell\in \Sigma$, by $\Sigma\setminus\{\ell\}$ we will understand the collection of linear forms of $\Sigma$ from which $\ell$ has been removed. Also, we denote $|\Sigma|=n$, and ${\rm rk}(\Sigma):={\rm ht}(\langle \ell_1,\ldots,\ell_n\rangle)$.

Let $1\leq a\leq n$ be an integer and define {\em the ideal generated by $a$-fold products of $\Sigma$} to be the ideal of $R$ $$I_a(\Sigma):=\langle \{\ell_{i_1}\cdots\ell_{i_a}|1\leq i_1<\cdots<i_a\leq n\}\rangle.$$ We also make the convention $I_0(\Sigma):=R$, and $I_b(\Sigma)=0$, for all $b>n$. Also, if $\Sigma=\emptyset$, $I_a(\Sigma)=0$, for any $a\geq 1$.

When the linear forms $\ell_1,\ldots,\ell_n$ define a hyperplane arrangement $\A\subset\mathbb P^{k-1}$ (so for all $i\neq j$, $\gcd(\ell_i,\ell_j)=1$), we will denote $\Sigma=\A=\{\ell_1,\ldots,\ell_n\}$.

\medskip

A homogeneous ideal $I\subset R$ generated in degree $d$ it is said to have {\em linear (minimal) graded free resolution}, if one has the graded free resolution $$0\rightarrow R^{n_{b+1}}(-(d+b))\rightarrow\cdots\rightarrow R^{n_2}(-(d+1))\rightarrow R^{n_1}(-d)\rightarrow R \rightarrow R/I\rightarrow 0,$$ for some positive integer $b$. The integers $n_j\geq 1$ are called the {\em betti numbers} of $R/I$, and since we have exactly one shift at each step in the resolution, they match the {\em graded betti numbers}. By convention, the zero ideal has linear graded free resolution. Also we say that $R/I$ has linear graded free resolution if and only if $I$ has linear graded free resolution.

\cite[Conjecture 1]{AnGaTo} states that for any collection of linear forms $\Sigma$, and any $1\leq a\leq|\Sigma|$, the ideals $I_a(\Sigma)$ have linear graded free resolution. In \cite{To3} is presented the current state of this conjecture, as well as new instances when the conjecture is true: Theorem 2.2 shows the validity for $k=2$, and Theorem 2.5 shows the validity for $a=|\Sigma|-2$.

In this article we check the validity of the conjecture when $a=d+1$, where $d$ is the minimum distance of the linear code dual to $\Sigma$; and for $k=3$, but when $\Sigma=\A=\{\ell_1,\ldots,\ell_n\}$.

The strategy to prove these two cases is the following.
\begin{itemize}
  \item The proof is by induction on pairs $(|\Sigma|, {\rm rk}(\Sigma))$.
  \item Often we appeal to the simple observation that if $\ell\in \Sigma$, then for all $a$, $I_a(\Sigma)=\ell\cdot I_{a-1}(\Sigma\setminus\{\ell\})+I_a(\Sigma\setminus\{\ell\})$.
  \item By knowing how $I_a(\Sigma)^{\rm sat}$ looks like, from inductive hypotheses coupled with Remark \ref{remark0} below, we are able to show that there exists an $\ell\in\Sigma$ such that $$I_a(\Sigma):\ell=I_{a-1}(\Sigma\setminus\{\ell\}).$$
  \item Consider the short exact sequence of graded $R$-modules:
$$0\longrightarrow \frac{R(-1)}{I_{a-1}(\Sigma\setminus\{\ell\})}\stackrel{\cdot\ell}\longrightarrow \frac{R}{I_a(\Sigma)}\longrightarrow \frac{R}{\langle\ell,I_a(\Sigma)\rangle}\longrightarrow 0.$$
  \item Most of the time due to inductive hypotheses, the leftmost and the rightmost nonzero modules have linear graded free resolutions, equivalently, by \cite[Theorem 1.2]{EiGo}, their Castelnuovo-Mumford regularities are both $\leq a-1$.
  \item Apply the inequalities of regularity under short exact sequence (see \cite[Corollary 20.19 b.]{Ei2}), to conclude that ${\rm reg}(R/I_a(\Sigma))\leq a-1$, and therefore that $R/I_a(\Sigma)$ has linear graded free resolution.
\end{itemize}     

\begin{rem}\label{remark0} Let $J\subset R$ be an ideal generated in degree $a$. Then $J\subseteq J^{\rm sat}\cap {\frak m}^a$.\footnote{If $I\subset R$ is a homogeneous ideal, by definition $I^{\rm sat}:=\{f\in R| \exists n(f)\geq 0 \mbox{ such that }{\frak m}^{n(f)}\cdot f\subset I\}$.} If $R/J$ has linear graded free resolution (equivalently, ${\rm reg}(R/J)=a-1$), since ${\rm H}_{\frak m}^0(R/J)=J^{\rm sat}/J$, by \cite[Theorem 4.3]{Ei}, we have $(J^{\rm sat}/J)_e=0, \mbox{ for any }e\geq a$. This means that $J^{\rm sat}\cap {\frak m}^a\subseteq J$, and therefore $$J=J^{\rm sat}\cap {\frak m}^a.$$
\end{rem}

In the last part we present an addition-deletion technique to find the graded betti numbers for $\mathbb K[x,y,z]/I_a(\A)$ for any $1\leq a\leq |\A|$, where $\A$ is a line arrangement in $\mathbb P^2$.

\section{The scheme of projective codewords of minimum weight}

A linear code $\mathcal C$ of dimension $k$, and length $n$, is the image of a $\mathbb K$-linear map $\mathbb K^k\longrightarrow \mathbb K^n$, given by a (generating) matrix $G$ of size $k\times n$; most of the time one supposes that $G$ has no zero column, and we suppose this as well. The minimum distance, $d_1(\mathcal C)$, is the minimum number of nonzero entries in a nonzero vector in $\mathcal C$. $\dim(\mathcal C)={\rm rk}(G)$, $n$, and $d_1(\mathcal C)$ are called the parameters of $\mathcal C$, and they are invariant under rescaling and permutation of the columns of $G$. The left-multiplication of $G$ by any $k\times k$ invertible matrix gives the same linear code $\mathcal C$. Because of these, one has the following duality:

To each column $(c_1,\ldots,c_k)^T$ of $G$ consider the dual linear form $c_1x_1+\cdots+c_kx_k$ in $R=\mathbb K[x_1,\ldots,x_k]$. Consequently, to $G$ we can associate a collection of linear forms $\Sigma=(\ell_1,\ldots,\ell_n)\subset R$. Also, this process is reversible: to any collection of linear forms we can associate a generating matrix of a linear code $\mathcal C$. The dimension of $\mathcal C$ equals the rank of the collection of linear forms dual to some (any) generating matrix of the code. Because of this duality, we will replace $\Sigma$ by its dual code $\mathcal C$.

Suppose ${\rm rk}(G)=k$. Generalizing the notion of minimum distance, for any $1\leq r\leq k$ one can define the $r$-th generalized Hamming weight, $d_r(\mathcal C)$, which by classical results in coding theory (see for example \cite[Corollary 1.3 and Proposition 1.7]{AnGaTo}) has the following description: $n-d_r(\mathcal C)$ is the maximum number of columns of $G$ that span a $k-r$ dimensional vector space. For example, $n-d_k(\mathcal C)=0$, since we assumed that $G$ has no zero columns. Or, $n-d_{k-1}(\mathcal C)$ is the maximum number of columns of $G$ that are proportional to each-other.

The $r$-th generalized Hamming weights help determine the heights of ideals generated by $a$-fold products of linear forms. From \cite[Proposition 2.2]{AnGaTo}, for $r=1,\ldots,k$, with the convention that $d_0(\mathcal C)=0$, for any $d_{r-1}(\mathcal C)<a\leq d_{r}(\mathcal C)$, one has
$${\rm ht}(I_a(\mathcal C))=k+1-r.$$ For example, if $1\leq a\leq d_1(\mathcal C)$, then ${\rm ht}(I_a(\mathcal C))=k$, the maximum possible value.

From \cite[Theorem 3.1]{To}, we have that for any $1\leq a\leq d_1(\mathcal C)$, $$I_a(\mathcal C)={\frak m}^a.$$ This has been the starting point of \cite[Conjecture 1]{AnGaTo}, since powers of the irrelevant ideal have linear graded free resolution (see for example \cite[Corollary 1.5]{EiGo}).

In this first part of the paper we will show that $I_{d_1(\mathcal C)+1}(\mathcal C)$ has also linear graded free resolution, generalizing one of the main results in \cite{AnTo}. As mentioned in there, such a result has very good applications in error-correcting messages, via saturation with respect with an extra variable.

\subsection{Projective codewords of minimum weight.} The weight of a vector ${\bf v}\in \mathbb K^n$, denoted ${\rm wt}({\bf v})$, is the number of its nonzero entries. So the product of any ${\rm wt}({\bf v})+1$ entries of ${\bf v}$ is zero, and there exist a product of ${\rm wt}({\bf v})$ entries of ${\bf v}$ that is nonzero.

Elements of $\mathcal C$ are called codewords, and a codeword of minimum weight is a codeword of weight $d_1(\mathcal C)$. A projective codeword of minimum weight is the equivalence class of a codeword of minimum weight, under multiplication by nonzero scalars. So we can think of a projective codeword of minimum weight as a point $[\beta_1,\ldots,\beta_n]\in\mathbb P^{n-1}$. Suppose ${\rm rk}(G)=k$. Then, each projective codeword of minimum weight corresponds uniquely to a point $[\alpha_1,\ldots,\alpha_k]\in\mathbb P^{k-1}$, via the left-multiplication $(\alpha_1,\ldots,\alpha_k)\cdot G=(\beta_1,\ldots,\beta_n)$.

These two observations put together give that the variety of projective codewords of minimum weight corresponds to $$V(I_{d_1(\mathcal C)+1}(\mathcal C))\subset\mathbb P^{k-1}.$$ We will abuse terminology by saying that the scheme of projective codewords of minimum weight is defined by $I_{d_1(\mathcal C)+1}(\mathcal C)$.

\begin{rem}\label{remark1} By \cite[Lemma 2.2]{To}, if $V(I_{d_1(\mathcal C)+1}(\mathcal C))=\{Q_1,\ldots,Q_m\}$, then
$$(I_{d_1(\mathcal C)+1}(\mathcal C))^{\rm sat}={\frak q_1}\cap\cdots\cap{\frak q_m},$$ where, for $1\leq i\leq m$, ${\frak q_i}\subset R$ is the linear prime ideal of the point $Q_i$.
\end{rem}

\begin{rem}\label{remark2} Let $\ell\in\Sigma$, and denote by $\mathcal C'$ the linear code dual to $\Sigma\setminus\{\ell\}$. Denote $d:=d_1(\mathcal C)$, and $d':=d_1(\mathcal C')$. Also denote by $G$, respectively $G'$, the corresponding dual generating matrices; $G'$ is the matrix obtained from $G$ by removing the column dual to $\ell$. The following classical results happen (see \cite{HuPl}):
\begin{itemize}
  \item[(i)] If $\dim(\mathcal C')=k-1$, then all the $n-1$ columns of $G'$ will span a $k-1$ dimensional vector space, so $d=1$.
  \item[(ii)] If $d\geq 2$, then $\dim(\mathcal C')=k$, and $d'$ equals $d$ or $d-1$.
  \item[(iii)] Let $d\geq 2$. Suppose $d'=d-1$, and suppose $\ell=\ell_n$ dual to the last column of $G$. Let ${\bf v'}=(\beta_1,\ldots,\beta_{d-1},0,\ldots,0)\in\mathbb K^{n-1}$ be a projective codeword of minimum weight of $\mathcal C'$, and suppose it corresponds to $\underbrace{(\alpha_1,\ldots,\alpha_k)}_{Q}\cdot G'={\bf v'}$. Obviously, ${\bf v}:=({\bf v'},\ell(Q))\in \mathbb K^n$ is a codeword of $\mathcal C$. If $\ell(Q)=0$, then ${\bf v}$ is a projective codeword of weight $d-1$ of $\mathcal C$; impossible. So $\ell(Q)\neq 0$, and ${\bf v}$ is a projective codeword of minimum weight (equal to $d$) of $\mathcal C$.
      So, if $d'=d-1$, then $$(I_{d'+1}(\mathcal C'))^{\rm sat}=\bigcap_{\ell\notin{\frak q_i}}{\frak q_i},$$ where $(I_{d+1}(\mathcal C))^{\rm sat}={\frak q_1}\cap\cdots\cap{\frak q_m}$.
\end{itemize}
\end{rem}

\begin{thm}\label{main1} Let $\Sigma$ be a collection of $n$ linear forms in $R:=\mathbb K[x_1,\ldots,x_k]$. Let $d$ be the minimum distance of the linear code dual to $\Sigma$. Then, $I_{d+1}(\Sigma)$ has linear graded free resolution.
\end{thm}
\begin{proof} We will prove the result by induction on the pairs $(|\Sigma|, {\rm rk}(\Sigma))$, where $|\Sigma|\geq {\rm rk}(\Sigma)\geq 2$.

\noindent{\bf Base Cases.} If ${\rm rk}(\Sigma)=2$, after a change of variables we can suppose that $\Sigma\subset \mathbb K[x_1,x_2]$. \cite[Theorem 2.2]{To3} proves this case.

If $|\Sigma|={\rm rk}(\Sigma)$ (which we can suppose it equals $k\geq 2$), then after a change of variables we can assume that $\Sigma=(x_1,\ldots,x_k)$. In this case $d=1$. From results concerning star configurations (see part (3) in the Introduction of \cite{To3}), $I_a(\Sigma)$ has linear graded free resolution for any $1\leq a\leq k$.

\medskip

\noindent{\bf Induction Step.} Suppose $|\Sigma|>{\rm rk}(\Sigma)=k\geq 3$. Let $\mathcal C$ be the linear code dual to $\Sigma$. So the length of $\mathcal C$ is $n$, the dimension is $k$, and the minimum distance is $d\geq 1$. The inductive hypotheses say that if $\ell\in\mathcal C$, and if $\mathcal C':=\mathcal C\setminus\{\ell\}$ has minimum distance $d'$, then $I_{d'+1}(\mathcal C')$ has linear graded free resolution.

\noindent CLAIM 1: There exists $\ell\in\Sigma$ such that, with the notations above, $$I_{d+1}(\mathcal C):\ell=I_d(\mathcal C').$$

\begin{proof} Suppose $d=1$. Then, there exists a codeword of $\mathcal C$ with exactly one nonzero entry (say, the $n$-th entry). Since this is a vector in $\mathbb K^n$ which is a linear combination of the $k$ rows of $G$, we can replace the last row of $G$ by this vector, and after a change of variables we can assume $\ell_n=x_k$. So, after a left-multiplication of $G$ by an invertible matrix, we can assume that $\ell_1,\ldots,\ell_{n-1}$ are linear forms in variables $x_1,\ldots,x_{k-1}$, and $\ell_n=x_k$.

Take $\ell=\ell_n=x_k$. Let $f\in I_2(\mathcal C):x_k$. Since $I_2(\mathcal C)=x_kI_1(\mathcal C')+I_2(\mathcal C')$, we have $x_k(f-g)\in I_2(\mathcal C')$, for some $g\in I_1(\mathcal C')$. But from the change of variables above, $x_k$ is a nonzero divisor mod $I_2(\mathcal C')$, so $f-g\in I_2(\mathcal C')$. Since $I_2(\mathcal C')\subset I_1(\mathcal C')$, we have $I_2(\mathcal C):x_k\subset I_1(\mathcal C')$. The inclusion $\supseteq$ is obvious, so for $d=1$ the claim is shown.

\medskip

Suppose $d\geq 2$. Let $\ell$ be any element of $\Sigma$. From Remark \ref{remark2} (ii), $\dim(\mathcal C')=k$, and $d'=d$ or $d'=d-1$.

If $d'=d$, then $I_d(\mathcal C')={\frak m}^d$. Since $I_{d+1}(\mathcal C)\subset {\frak m}^{d+1}$, we have that $$I_{d+1}(\mathcal C):\ell\subset {\frak m}^{d+1}:\ell={\frak m}^d=I_d(\mathcal C').$$ Since the other inclusion is obvious, the claim is shown for this situation.

Suppose $d'=d-1$. We have that $I_{d+1}(\mathcal C)\subseteq (I_{d+1}(\mathcal C))^{\rm sat}\cap{\frak m}^{d+1}$. Therefore, from Remark \ref{remark1}, $$I_{d+1}(\mathcal C):\ell\subseteq ({\frak q_1}:\ell)\cap\cdots\cap({\frak q_m}:\ell)\cap({\frak m}^{d+1}:\ell).$$ If $\ell\in {\frak q_i}$, then ${\frak q_i}:\ell=R$, and if $\ell\notin {\frak q_j}$, then ${\frak q_j}:\ell={\frak q_j}$. Also, ${\frak m}^{d+1}:\ell={\frak m}^d$. All together, via Remark \ref{remark2} (iii), give $$I_{d+1}(\mathcal C):\ell\subseteq (I_{d'+1}(\mathcal C'))^{\rm sat}\cap {\frak m}^d.$$ But, by inductive hypotheses, $I_{d'+1}(\mathcal C')$ has linear graded free resolution, and therefore, by Remark \ref{remark0}, the last intersection above equals $I_{d'+1}(\mathcal C')$ itself. So in this case we obtained that $$I_{d+1}(\mathcal C):\ell\subseteq I_{d}(\mathcal C').$$ Since the other inclusion is obvious, we proved the claim for this remaining case as well.
\end{proof}

Let $\ell\in\Sigma$ be such that via CLAIM 1, $I_{d+1}(\mathcal C):\ell=I_d(\mathcal C')$. Suppose $\ell=\ell_n$.

As we explained in the Introduction, we have the short exact of graded $R$-modules

$$0\longrightarrow \frac{R(-1)}{I_d(\mathcal C')}\longrightarrow \frac{R}{I_{d+1}(\mathcal C)}\longrightarrow \frac{R}{\langle\ell, I_{d+1}(\mathcal C)\rangle}\longrightarrow 0.$$

First we deal with the leftmost nonzero module.
\begin{itemize}
  \item[(a)] If $d=1$, then $I_1(\mathcal C')=\langle\ell_1,\ldots,\ell_{n-1}\rangle$, which is a linear prime ideal, hence it has a linear graded free resolution.
  \item[(b)] If $d\geq 2$ (so $\dim(\mathcal C')=k$), and $d'=d$, then $I_d(\mathcal C')={\frak m}^d$, which has linear graded free resolution.
  \item[(c)] If $d\geq 2$ and $d'=d-1$, then $I_d(\mathcal C')=I_{d'+1}(\mathcal C')$, which by inductive hypotheses has linear graded free resolution.
\end{itemize} In conclusion, $${\rm reg}\left(\frac{R(-1)}{I_d(\mathcal C')}\right)=(d-1)+1=d.$$

\medskip

Now we deal with the rightmost nonzero module. We can suppose, after a change of variables, that $\ell=\ell_n=x_k$. Suppose that for some $m=0,\ldots,n-1$, after a permutation of the elements of $\Sigma$, and a rescaling, we have $\ell_{m+1}=\cdots=\ell_n=x_k$, and that for all $1\leq i\leq m$, $\ell_i=\bar{\ell_i}+c_ix_k, c_i\in\mathbb K$, where $\bar{\ell_i}$ are linear forms in variables $x_1,\ldots,x_{k-1}$.

Consider $\bar{\mathcal C}$ the linear code dual to $\bar{\Sigma}:=(\bar{\ell_1},\ldots,\bar{\ell_m})\subset \bar{R}:=\mathbb K[x_1,\ldots,x_{k-1}]$. Let $\bar{G}$ be the corresponding generating matrix. Since we assumed that ${\rm rk}(\Sigma)=k$, then ${\rm rk}(\bar{\Sigma})={\rm rk}(\bar{G})= \dim(\bar{\mathcal C})=k-1$. The length of $\bar{\mathcal C}$ is $m$. Let $\bar{d}:=d_1(\bar{\mathcal C})$ be the minimum distance of $\bar{\mathcal C}$.

We have that $u:=m-\bar{d}$ is the maximum number of columns of $\bar{G}$ that span an $(k-1)-1=k-2$ vector space. WLOG, suppose these are the first $u$ columns of $\bar{G}$. Then ${\rm ht}(\langle \bar{\ell_1},\ldots,\bar{\ell_u}\rangle)=k-2$. Since $x_k$ is a nonzero divisor mod $\langle x_1,\ldots,x_{k-1}\rangle$, we have ${\rm ht}(\langle\bar{\ell_1},\ldots,\bar{\ell_u}, x_k\rangle)=k-1$. But

$$\langle\bar{\ell_1},\ldots,\bar{\ell_u}, x_k\rangle=\langle\ell_1,\ldots,\ell_u, x_k\rangle=\langle\ell_1,\ldots,\ell_u, \ell_{m+1},\ldots,\ell_n\rangle.$$ So the first $u$ and the last $n-m$ columns of $G$ span a $k-1$ dimensional vector space. Therefore
$$n-d\geq u+n-m=n-\bar{d},$$ leading to $\bar{d}\geq d$.

\begin{itemize}
  \item[(1)] If $\bar{d}\geq d+1$, then $I_{d+1}(\bar{\mathcal C})=\langle x_1,\ldots,x_{k-1}\rangle^{d+1}$, which has a linear graded free resolution.
  \item[(2)] If $\bar{d}=d$, then $I_{d+1}(\bar{\mathcal C})=I_{\bar{d}+1}(\bar{\mathcal C})$, which by inductive hypotheses has a linear graded free resolution.
\end{itemize} To sum up we got $${\rm reg}\left(\frac{\bar{R}}{I_{d+1}(\bar{\mathcal C})}\right)=d.$$

But $R/\langle\ell,I_{d+1}(\mathcal C)\rangle$ and $\bar{R}/I_{d+1}(\bar{\mathcal C})$ are isomorphic as $R$-modules, so they have the same regularity (see \cite[Corollary 4.6]{Ei}).

\medskip

As we explained in the Introduction, we obtain that $I_{d+1}(\mathcal C)$ has linear graded free resolution.
\end{proof}

\section{Line arrangements in $\mathbb P^2$}

In this section we assume that $\Sigma=(\ell_1,\ldots,\ell_n)\subset R:=\mathbb K[x,y,z]$, with $\gcd(\ell_i,\ell_j)=1$, for $i\neq j$. So $\ell_1,\ldots,\ell_n$ define a line arrangement $\A\subset\mathbb P^2$; we will write $\Sigma=\A=\{\ell_1,\ldots,\ell_n\}\subset R$. Suppose ${\rm rk}(\A)=3$.

We will show that for any $1\leq a\leq n$, the ideal $I_a(\A)$ has linear graded free resolution. This is known to be true in various instances:
\begin{itemize}
  \item[(a)] If $a=n$, then $I_n(\A)=\langle\ell_1\cdots\ell_n\rangle$, which has linear graded free resolution.
  \item[(b)] If $a=n-1$, then, by \cite{Sc}, $0\longrightarrow R(-n)^{n-1}\longrightarrow R(-n+1)^n\longrightarrow I_{n-1}(\A)\longrightarrow 0$ is a linear graded free resolution.
  \item[(c)] If $a=n-2$, then \cite[Theorem 2.4]{To3} shows that $I_{n-2}(\A)$ has linear graded free resolution.
  \item[(d)] Let $\mathcal C$ be the linear code dual to $\A$. Let $d:=d_1(\mathcal C)$. Then, as we observed, $n-d$ is the maximum number of columns of the generating matrix $G$ that span a $3-1=2$ dimensional vector space. So $m:=n-d$ is the maximum number of concurrent lines of $\A$. If $1\leq a\leq n-m$, then $I_a(\A)=\langle x,y,z\rangle^a$, which has linear graded free resolution.
  \item[(e)] If $a=n-m+1$, by Theorem \ref{main1}, $I_{n-m+1}(\A)$ has linear graded free resolution.
  \item[(f)] If $\A$ is generic (i.e., if any three linear forms of $\A$ are linearly independent), then for any $1\leq a\leq n$, $I_a(\A)$ has linear graded free resolution. This is true from parts (a), (b), and (d) above, since $m=2$.
\end{itemize}

Suppose $Sing(\A):=\{P_1,\ldots,P_s\}\subset\mathbb P^2$ is the set of all the intersection points of the lines of $\A$. For $i=1,\ldots,s$, let ${\frak p_i}\subset R$ be the defining ideal of $P_i$; also, let $n_i=n(P_i)$ be the number of lines of $\A$ that intersect at the point $P_i$. Note that $m=\max\{n_1,\ldots,n_s\}$, and $n_i\geq 2$ for all $i=1,\ldots,s$. If $\A$ is generic, then $n_i=2$ for all $i=1,\ldots,s$.

\begin{lem}\label{lemma2} For any $1\leq b\leq m-1$ we have
$$(I_{n-b}(\A))^{\rm sat}={\frak p_1}^{n_1-b}\cap\cdots\cap {\frak p_s}^{n_s-b},$$ where if $n_j-b\leq 0$, then ${\frak p_j}^{n_j-b}$ is by convention equal to $R$.
\end{lem}
\begin{proof} Since $1\leq b\leq m-1$, then $n-m+1\leq n-b\leq n-1$. If $\mathcal C$ is the linear code dual to $\A$, we have $d_1(\mathcal C)=n-m$, and $d_2(\mathcal C)=n-1$. So $d_1(\mathcal C)<n-b\leq d_2(\mathcal C)$, which leads to ${\rm ht}(I_{n-b}(\A))=2$. Since for all $j=1,\ldots,s$, $(n-b)-n+n_j=n_j-b$, and since $I_{n-b}(\A)=(I_{n-b}(\A))^{\rm sat}\cap J$, where $J$ is an ideal which is $\langle x,y,z\rangle$-primary, \cite[Proposition 2.3]{AnGaTo} proves the result.
\end{proof}

\begin{thm}\label{main2} Let $\A=\{\ell_1,\ldots,\ell_n\}\subset R:=\mathbb K[x,y,z]$ be a line arrangement of rank 3. Then, for any $1\leq a\leq n$, the ideal $I_a(\A)$ has linear graded free resolution.
\end{thm}
\begin{proof} If $m=\max\{n_1,\ldots,n_s\}$, from parts (a) and (d) above we can suppose $n-m+1\leq a\leq n-1$. Let $b:=n-a$, so $1\leq b\leq m-1$. Since ${\rm rk}(\A)=3$, then $m<n$, so $b\leq n-2$.

We will show that $I_{n-b}(\A)$ has linear graded free resolution, by induction on $n=|\A|\geq 3$.

\medskip

\noindent{\bf Base Case.} If $n=3$, then part (f) above proves the claim.

\medskip

\noindent{\bf Inductive Step.} Suppose $n>3$. Let $\ell\in \A$, and let $\A'=\A\setminus\{\ell\}$.

\noindent CLAIM 2: We have $I_{n-b}(\A):\ell=I_{n-b-1}(\A')$.

\begin{proof} If ${\rm rk}(\A')=2$, after a change of variables we can suppose $\ell=\ell_n=z$, and $\A'=\{\ell_1,\ldots,\ell_{n-1}\}\subset\mathbb K[x,y]$. Same argument as in the beginning of proof of CLAIM 1, we have $$I_{n-b}(\A):z=I_{n-b-1}(\A')=\langle x,y\rangle^{n-b-1}.$$ The last equality comes from $1\leq n-b-1\leq n-2$, and from the fact that the minimum distance of the linear code dual to $\A'$ is $n-1-1=n-2$.

Suppose ${\rm rk}(\A')=3$. Suppose $\ell$ passes through $P_1,\ldots,P_r$, and avoids $P_{r+1},\ldots,P_s$. Since ${\rm rk}(\A)=3$, then $r\geq 2$. We have $I_{n-b}(\A)\subseteq (I_{n-b}(\A))^{\rm sat}\cap \langle x,y,z\rangle^{n-b}$. Therefore

$$I_{n-b}(\A):\ell\subseteq ({\frak p_1}^{n_1-b}:\ell)\cap\cdots\cap ({\frak p_s}^{n_s-b}:\ell)\cap\underbrace{\langle x,y,z\rangle^{n-b}:\ell}_{\langle x,y,z\rangle^{n-b-1}}.$$

If $\ell\in {\frak p_j}$, then ${\frak p_j}^{n_j-b}:\ell={\frak p_j}^{n_j-b-1}$; and if $\ell\notin {\frak p_j}$, then ${\frak p_j}^{n_j-b}:\ell={\frak p_j}^{n_j-b}$.

For $i=1,\ldots, r$, there are $n_i-1$ lines of $\A'$ passing through $P_i$, and for $j=r+1,\ldots, s$, there are $n_j$ lines of $\A'$ passing through $P_j$. If for some $i\in\{1,\ldots,r\}$, $n_i=2$, then $P_i$ doesn't show up in $Sing(\A')$. But at the same time, $n_i-b-1\leq 0$, so by Lemma \ref{lemma2} we can conclude $$({\frak p_1}^{n_1-b}:\ell)\cap\cdots\cap ({\frak p_s}^{n_s-b}:\ell)={\frak p_1}^{n_1-b-1}\cap\cdots\cap {\frak p_r}^{n_r-b-1}\cap {\frak p_{r+1}}^{n_{r+1}-b}\cap\cdots\cap {\frak p_s}^{n_s-b}=(I_{n-1-b}(\A'))^{\rm sat}.$$

By inductive hypotheses and Remark \ref{remark0}, $I_{n-1-b}(\A')=(I_{n-1-b}(\A'))^{\rm sat}\cap\langle x,y,z\rangle^{n-1-b}$, so we obtained
$$I_{n-b}(\A):\ell\subseteq I_{n-1-b}(\A').$$ Since the other inclusion is obvious, the claim is shown.
\end{proof}

The equality in CLAIM 2 gives the short exact sequence:

$$0\longrightarrow R(-1)/I_{n-1-b}(\A')\longrightarrow R/I_{n-b}(\A)\longrightarrow R/\langle \ell,I_{n-b}(\A)\rangle\longrightarrow 0.$$

We can suppose $\ell=\ell_n=z$. Then $\langle \ell,I_{n-b}(\A)\rangle=\langle z,I_{n-b}(\bar{\A})\rangle$, where $\bar{\A}=(\bar{\ell_1},\ldots,\bar{\ell_{n-1}})\subset\mathbb K[x,y]$, and $\ell_i\equiv \bar{\ell_i}\mbox{ mod }z$. Some of the linear forms of $\bar{\A}$ may be proportional. Nonetheless, by \cite[Theorem 2.2]{To3}, $I_{n-b}(\bar{\A})$ has linear graded free resolution, and so does the rightmost nonzero module in the short exact sequence above.

About the leftmost nonzero module in the short exact sequence, if ${\rm rk}(\A')=2$, we saw at the beginning of the proof of CLAIM 2, that $I_{n-1-b}(\A')$ has linear graded free resolution. If ${\rm rk}(\A')=3$, then by inductive hypotheses, $I_{n-1-b}(\A')$ has linear graded free resolution. Therefore, by the strategy mentioned in the Introduction, $I_{n-b}(\A)$ also has linear graded free resolution.
\end{proof}

\section{Graded betti numbers}

\subsection{The case $k=2$.} \label{k2} First suppose $\Sigma=(\underbrace{\ell_1,\ldots,\ell_1}_{m_1},\ldots,\underbrace{\ell_t,\ldots,\ell_t}_{m_t})\subset \overline{R}:=\mathbb K[x,y]$, with $t\geq 1$, and $\gcd(\ell_i,\ell_j)=1$. Let $1\leq a\leq v:=m_1+\cdots+m_t$, and for $i=1,\ldots, t$, let $d_i:=\max\{m_i+a-v,0\}$. Then, by \cite[Proposition 2.3]{AnGaTo} and \cite[Lemma 2.1]{To3}\footnote{In the pre-published version, this lemma states that if $n_1=\max\{n_1,\ldots,n_m\}$, then, for $1\leq a\leq n-n_1$, $I_a(\Sigma)=\langle x,y\rangle^a$, and for $n\geq a\geq n-n_1+1$, $I_a(\Sigma)=\ell_1I_{a-1}(\Sigma\setminus\{\ell_1\})$.},
$$I_a(\Sigma)=\ell_1^{d_1}\cdots\ell_t^{d_t}\cdot\langle x,y\rangle^e,$$ where $e=\max\{a-(d_1+\cdots+d_t),0\}$. Then, the minimal graded free resolution of $\overline{R}/I_a(\Sigma)$ is $$0\longrightarrow \overline{R}^e(-(a+1))\longrightarrow\overline{R}^{e+1}(-a)\longrightarrow \overline{R};$$ so the graded betti numbers are $b_1(a,\Sigma)=e+1$, $b_2(a,\Sigma)=e$, and $b_3(a,\Sigma)=0$.

\subsection{The case $k=3$.} Let $\A=\{\ell_1,\ldots,\ell_n\}\subset R:=\mathbb K[x,y,z]$ be a line arrangement in $\mathbb P^2$, of rank 3. Let ${\frak m}:=\langle x,y,z\rangle$.

Suppose $Sing(\A):=\{P_1,\ldots,P_s\}\subset\mathbb P^2$ is the set of all the intersection points of the lines of $\A$. For $i=1,\ldots,s$, let ${\frak p_i}\subset R$ be the defining ideal of $P_i$; also, let $n_i=n(P_i)\geq 2$ be the number of lines of $\A$ that intersect at the point $P_i$, and let $m:=\max\{n_1,\ldots,n_s\}$.

For $1\leq a\leq n$, let

$$0\longrightarrow R^{b_3(a,\A)}(-(a+2))\longrightarrow R^{b_2(a,\A)}(-(a+1))\longrightarrow R^{b_1(a,\A)}(-a)\longrightarrow R$$ be the minimal (linear) graded free resolution of $R/I_a(\A)$.

\begin{itemize}
  \item[(i)] As we already discussed, if $1\leq a\leq n-m$, then $I_a(\A)={\frak m}^a$, and by \cite[Proposition 1.7(c)]{EiGo}, the graded betti numbers of $R/I_a(\A)$ are $$b_i(a,\A)={{a+2}\choose{a+i-1}}\cdot{{a+i-2}\choose{a-1}}, i=1,2,3.$$
  \item[(ii)] If $a=n$, then $I_n(\A)=\langle \ell_1\cdots\ell_n\rangle$, and the graded betti numbers of $R/I_n(\A)$ are
  $$b_1(n,\A)=1, b_2(n,\A)=b_3(n,\A)=0.$$
  \item[(iii)] If $a=n-1$, by \cite[Lemma 3.2]{Sc}, the graded betti numbers of $R/I_{n-1}(\A)$ are
  $$b_1(n-1,\A)=n, b_2(n-1,\A)=n-1, b_3(n,\A)=0.$$
  \item[(iv)] If $a=n-2$, \cite[Theorem 2.4]{To3} presents the minimal free resolution of $R/I_{n-2}(\A)$, and hence its graded betti numbers.
  \item[(v)] If $\A$ is generic, then $m=2$, and therefore items (i), (ii), and (iii) give the graded betti numbers of $R/I_a(\A)$ for any $1\leq a\leq n$.
\end{itemize}

\begin{rem}\label{remark_not_generic} From the itemized list above, what is left to analyze are the graded betti numbers of $R/I_a(\A)$, for $n-m+1\leq a\leq n-2$, and when $\A$ is not generic. In this case we have ${\rm ht}(I_a(\A))=2$. We also have that $R/I_a(\A)$ is not Cohen-Macaulay. To show this we use the same trick as in proof of \cite[Proposition 2.4]{To}: let $\ell_1,\ell_2,\ell_3\in\A$ be such that $\langle \ell_1,\ell_2,\ell_3\rangle={\frak m}$. Let $V(\ell)$ be some line that does not pass through $P_1,\ldots,P_s$. So $\ell$ is a nonzero divisor mod $(I_a(\A))^{\rm sat}$. Also $\ell=c_1\ell_1+c_2\ell_2+c_3\ell_3$, for some constants $c_1,c_2,c_3\in\mathbb K$.

Since $a-1\leq n-3$, there exist $\ell_{i_1},\ldots,\ell_{i_{a-1}}\in\A\setminus\{\ell_1,\ell_2,\ell_3\}$. Then
$$\ell\cdot(\ell_{i_1}\cdots\ell_{i_{a-1}})=c_1\ell_1\cdot(\ell_{i_1}\cdots\ell_{i_{a-1}})+c_2\ell_2\cdot(\ell_{i_1}\cdots\ell_{i_{a-1}})+ c_3\ell_3\cdot(\ell_{i_1}\cdots\ell_{i_{a-1}})\in I_a(\A)\subset (I_a(\A))^{\rm sat}.$$ Since $\ell$ is a nonzero divisor, we obtain $\ell_{i_1}\cdots\ell_{i_{a-1}}\in(I_a(\A))^{\rm sat}$, which is an element of degree $a-1$. So $I_a(\A)\subsetneq(I_a(\A))^{\rm sat}$.

To conclude the remark we have that in these conditions $b_3(a,\A)\geq 1$.
\end{rem}

\medskip

Let $\ell\in\A$, and consider $\A':=\A\setminus\{\ell\}$. If ${\rm rk}(\A')=3$, by Theorem \ref{main2}, $I_{a-1}(\A')$ has linear graded free resolution, CLAIM 2 in the proof of that theorem is satisfied: $$I_a(\A):\ell=I_{a-1}(\A').$$ If ${\rm rk}(\A')=2$, then CLAIM 2 is satisfied as we can see at the beginning of its proof.

As we have seen over and over again, this leads to the short exact sequence

$$0\longrightarrow R(-1)/I_{a-1}(\A')\longrightarrow R/I_a(\A)\longrightarrow R/\langle\ell, I_a(\A)\rangle\longrightarrow 0.$$

\medskip

$\bullet$ The leftmost nonzero module has minimal graded free resolution
$$0\longrightarrow R^{b_3(a-1,\A')}(-(a+2))\longrightarrow R^{b_2(a-1,\A')}(-(a+1))\longrightarrow R^{b_1(a-1,\A')}(-a)\longrightarrow R(-1).$$

\medskip

$\bullet$ Concerning the rightmost nonzero module, suppose $\ell=z$ and suppose $\ell(P_1)=\cdots=\ell(P_t)=0$ and $\ell(P_j)\neq 0, j=t+1,\ldots,s$. Then $$\langle \ell,I_a(\A)\rangle=\langle z, I_a(\overline{\A})\rangle,$$ where $\displaystyle\overline{\A}:= (\underbrace{\bar{\ell_1},\ldots,\bar{\ell_1}}_{n_1-1},\ldots,\underbrace{\bar{\ell_t},\ldots,\bar{\ell_t}}_{n_t-1})\subset\overline{R}:=\mathbb K[x,y]$. Then, the rightmost nonzero module has minimal free resolution
$$0\longrightarrow R^{b_2(a,\overline{\A})}(-(a+2))\longrightarrow R^{b_1(a,\overline{\A})+b_2(a,\overline{\A})}(-(a+1))\longrightarrow R(-1)\oplus R^{b_1(a,\overline{\A})}(-a)\longrightarrow R,$$ where $b_1(a,\overline{\A})$ and $b_2(a,\overline{\A})$ are determined according to Subsection \ref{k2}.

\medskip

Applying {\em mapping cone} technique, after appropriate cancellations (or by using loosely \cite[Lemma 1.13]{EiGo}), we obtain
\begin{itemize}
  \item[(A)] $b_3(a,\A)=b_3(a-1,\A')+b_2(a,\overline{\A})$.
  \item[(B)] $b_2(a,\A)=b_2(a-1,\A')+b_2(a,\overline{\A})+b_1(a,\overline{\A})$.
  \item[(C)] $b_1(a,\A)=b_1(a-1,\A')+b_1(a,\overline{\A})$.
\end{itemize}

\begin{exm} Consider the line arrangement $\A=\{x,x-z,x+z,z,y,y-z\}$. We have $n=6$ and $m=4$, so we are left to analyze the graded betti numbers when $a=3$, and $a=4$.

To be as explicit as possible, for each of these two case we will pick different sets of lines that will be recursively deleted.

\medskip

\noindent$\boxed{a=3}$

$\bullet$ Take $\ell=y$. Then $\A'=\{x,x-z,x+z,z,y-z\}$, and $\overline{\A}\approx(z,z,x+z,x,x-z)$ with $b_1(3,\overline{\A})=4$, $b_2(3,\overline{\A})=3$. From (A), (B), (C) we have
\begin{eqnarray}
b_3(3,\A)&=&b_3(2,\A')+3\nonumber\\
b_2(3,\A)&=&b_2(2,\A')+3+4\nonumber\\
b_1(3,\A)&=&b_1(2,\A')+4.\nonumber
\end{eqnarray}

$\bullet$ Take $\ell'=y-z$. Then $\A'':=\A'\setminus\{\ell'\}=\{x,x-z,x+z,z\}$, and $\overline{\A'}\approx(x,x+y,x-y,y)$. Then we have $b_1(1,\A'')=2$, $b_2(1,\A'')=1$, $b_3(1,\A'')=0$, and $b_1(2,\overline{\A'})=3$, $b_2(2,\overline{\A'})=2$. From (A), (B), (C) we have

\begin{eqnarray}
b_3(2,\A')&=&0+2=2\nonumber\\
b_2(2,\A')&=&1+2+3=6\nonumber\\
b_1(2,\A')&=&2+3=5,\nonumber
\end{eqnarray} leading to $$b_1(3,\A)=9,\, b_2(3,\A)=13,\, b_3(3,\A)=5.$$

\medskip

\noindent$\boxed{a=4}$

$\bullet$ Take $\ell=z$. Then $\A'=\{x,x-z,x+z,y-z,y\}$, and $\overline{\A}\approx(x,x,x,y,y)$ with $b_1(4,\overline{\A})=2$, $b_2(4,\overline{\A})=1$. From (A), (B), (C) we have
\begin{eqnarray}
b_3(4,\A)&=&b_3(3,\A')+1\nonumber\\
b_2(4,\A)&=&b_2(3,\A')+1+2\nonumber\\
b_1(4,\A)&=&b_1(3,\A')+2.\nonumber
\end{eqnarray}

$\bullet$ Take $\ell'=x$. Then $\A'':=\A'\setminus\{\ell'\}=\{x-z,x+z,y-z,y\}$, and $\overline{\A'}\approx(z,z,y-z,y)$. We have that $\A''$ is generic, with $1\leq 2\leq 4-2$, so from (i) above we have $b_1(2,\A'')=6$, $b_2(2,\A'')=8$, $b_3(2,\A'')=3$. Also we have $b_1(3,\overline{\A'})=3$, $b_2(3,\overline{\A'})=2$. From (A), (B), (C) we have

\begin{eqnarray}
b_3(3,\A')&=&3+2=5\nonumber\\
b_2(3,\A')&=&8+2+3=13\nonumber\\
b_1(3,\A')&=&6+3=9,\nonumber
\end{eqnarray} leading to $$b_1(4,\A)=11,\, b_2(4,\A)=16,\, b_3(4,\A)=6.$$ Observe that we obtained, as expected, the same graded betti numbers in \cite[Theorem 2.4]{To3} (with the notations of that theorem, we have $n=6, a=n-2=4, m=15, p(\A)=3+1=4$).
\end{exm}

%\noindent {\bf Acknowledgment.} \, The author is grateful to the anonymous referee for useful suggestions, comments, and corrections that improved the quality of the article.

\bigskip

%%%%%%%%%%%%%%%%%%%%%%%%%%%%%%%%%%%%%%%%%%%%%%%%%%%%%%%%
% Back to single space
\renewcommand{\baselinestretch}{1.0}
\small\normalsize % to get previous line to take
%%%%%%%%%%%%%%%%%%%%%%%%%%%%%%%%%%%%%%%%%%%%%%%%%%%%%%%%

\bibliographystyle{amsalpha}

\end{document}